\documentclass[11pt,a4paper]{article}
\usepackage[affil-it]{authblk}
\usepackage[utf8]{inputenc}
\usepackage{amsmath, amssymb, mathrsfs, amsthm, amsfonts, stmaryrd}
\usepackage{latexsym}
\usepackage{a4wide}

\newtheorem{theorem}{Theorem}[section]
\newtheorem{corollary}[theorem]{Corollary}
\newtheorem{lemma}[theorem]{Lemma}
\newtheorem{proposition}[theorem]{Proposition}

\theoremstyle{definition}
\newtheorem{definition}[theorem]{Definition}
\theoremstyle{remark}
\newtheorem{example}{Example}

\title{The Statistic $\mathtt{pinv}$ for Number System}

\author{Patrick Rabarison \thanks{P. Rabarison \\ Université d'Antananarivo, Département de Mathématiques et Informatique, 101 Antananarivo, Madagascar \\
e-mail: \texttt{prabarison@gmail.com}}
, 
Hery Randriamaro \thanks{H. Randriamaro (Corresponding Author)\\ Université d'Antananarivo, Département de Mathématiques et Informatique, 101 Antananarivo, Madagascar \\
e-mail: \texttt{hery.randriamaro@outlook.com}}}

\begin{document}

\maketitle

\begin{abstract}
\noindent The number of inversions is a statistic on permutation groups measuring the degree to which the entries of a permutation are out of order. We provide a generalization of that statistic by introducing the statistic number of pseudoinversions on the colored permutation groups. The main motivation to investigate that statistic is the possibility to use it to define a number system and a numeral system on the colored permutation groups. By means of the statistic number of $i$-pseudoinversions, we construct our number system, and a bijection between the set of positive integers and the colored permutation groups.

\bigskip 


\smallskip

\end{abstract}

\section{Introduction} \label{intro}

\noindent A statistic over a group is a function from that group to the set of nonnegative integers. One of the most studied statistics is the number of inversions on the symmetric group $\mathfrak{S}_n$ defined by $\mathtt{inv}\,\sigma := \#\big\{(i,j) \in [n]^2\ |\ i<j,\, \sigma(i)>\sigma(j)\big\}$. A well-known result is the equidistribution of $\mathtt{inv}$ with the statistic major index proved by Foata \cite{Fo}. In this article, we give a generalization of this statistic on the colored permutation group in order to create a more general code. The colored permutation group of $m$ colors and $n$ elements is the wreath product $\mathbb{U}_m \wr \mathfrak{S}_n$ of the group $\mathbb{U}_m$ of all $m^{th}$ roots of unity by the symmetric group $\mathfrak{S}_n$ on $[n]$. We represent an element $\pi \in \mathbb{U}_m \wr \mathfrak{S}_n$ by $$\pi = \left( \begin{array}{cccc} 1 & 2 & \dots & n\\
\zeta_{k_1} \sigma(1) & \zeta_{k_2} \sigma(2) & \dots & \zeta_{k_n} \sigma(n) \end{array} \right)\ \text{with}\ \sigma \in \mathfrak{S}_n\ \text{and}\ \zeta_{k_j} = e^{2\pi i \frac{k_j}{m}}.$$

\noindent For two integers $i<j$, let $[i,j] := \{i, i+1, \dots, j\}$.

\begin{definition}
The \emph{number of $i$-pseudoinversions} $\mathtt{pinv}_i\,\pi$ of a colored permutation $\pi \in \mathbb{U}_m \wr \mathfrak{S}_n$ is
$$\mathtt{pinv}_i\,\pi := k_i(n-i+1) + \big|\{j \in [i+1,n]\ |\ \sigma(i) > \sigma(j)\}\big|.$$
And the \emph{number of pseudoinversions} $\mathtt{pinv}\,\pi$ of a colored permutation $\pi$ in $\mathbb{U}_m \wr \mathfrak{S}_n$ is
$$\mathtt{pinv}\,\pi := \sum_{i=1}^n \mathtt{pinv}_i\,\pi.$$
\end{definition}

\begin{example} \label{Ex1}
Consider the element $\pi = \left( \begin{array}{cccc} 1 & 2 & 3 & 4 \\
2 & \zeta_1 1 & \zeta_4 4 & 3 \end{array} \right) \in \mathbb{U}_5 \wr \mathfrak{S}_4$. We have $\mathtt{pinv}_1\,\pi = 1$, $\mathtt{pinv}_2\,\pi = 3$, $\mathtt{pinv}_3\,\pi = 9$, $\mathtt{pinv}_4\,\pi = 0$, and $\mathtt{pinv}\,\pi = 13$.
\end{example}

\noindent The interest for investigating statistics on $\mathbb{U}_m \wr \mathfrak{S}_n$ recently arised. Bagno et al., for example, introduced the statistics $(c,d)$-descents and computed their distributions \cite[Proposition 1.1.]{BaGaMa}. We construct a number system by means of the cardinality $|\mathbb{U}_m \wr \mathfrak{S}_n|$, and the statistic $\mathtt{pinv}_i$.

\begin{definition}
Let $\mathsf{A} = (\mathsf{A}_i)_{i \in \mathbb{N}}$, $\mathsf{a} = (\mathsf{a}_i)_{i \in \mathbb{N}}$ be two sequences of positive integers. The pair $(\mathsf{A}, \mathsf{a})$ is a number system if, for every integer $n \in \mathbb{N}$, there exists $k \in \mathbb{N}$ such that $\displaystyle n = \sum_{i=0}^{k} \alpha_i \mathsf{A}_i$ with $\alpha_i \in [0, \mathsf{a}_i]$, and this representation in terms of $\mathsf{A}_i$'s and $\alpha_i$'s is unique.
\end{definition}

\noindent Using formal power series, Cantor provided a condition for a pair of positive integer sequences to be a number system \cite[§.2.]{Ca}. We provide a more suitable condition for a pair of positive integer sequences to be a number system.

\begin{proposition}\label{PrNu}
Let $\mathsf{A} = (\mathsf{A}_i)_{i \in \mathbb{N}}$ and $\mathsf{a} = (\mathsf{a}_i)_{i \in \mathbb{N}}$ be two sequences of strictly positive integers. The pair of sequences $(\mathsf{A}, \mathsf{a})$ is a number system if and only if $\mathsf{A}_0 = 1$ and
$$\mathsf{A}_k = \prod_{i=0}^{k-1}(1+ \mathsf{a}_i).$$
\end{proposition}

\noindent We prove Proposition \ref{PrNu} in Section \ref{SeNuSy}. 

\begin{corollary}
The pair of sequences $(\mathsf{G}, \mathsf{g}) = \big((\mathsf{G}_i)_{i \in \mathbb{N}}, (\mathsf{g}_i)_{i \in \mathbb{N}}\big)$ defined by
$$\mathsf{G}_i = m^i i! \quad \text{and} \quad \mathsf{g}_i = m(i+1)-1$$
is a number system.
\end{corollary}

\begin{proof}
We obtain the equality of Proposition \ref{PrNu} from $$\prod_{i=0}^{k-1}(1+ \mathsf{g}_i) = \prod_{i=0}^{k-1} m(i+1) = m^k k! = \mathsf{G}_i.$$
\end{proof}

\noindent We particularly use the number system $(\mathsf{G}, \mathsf{g})$ in this article. The number $\mathsf{G}_i$ is the cardinality of $\mathbb{U}_m \wr \mathfrak{S}_i$, and $\mathsf{g}_i$ is the maximal value of $\mathtt{pinv}_1$ relating to $\mathbb{U}_m \wr \mathfrak{S}_{i+1}$. The factorial number system introduced by Laisant \cite{La} corresponds to the case $m = 1$ of $(\mathsf{G}, \mathsf{g})$. We also construct a numeral system by means of the groups $\mathbb{U}_m \wr \mathfrak{S}_n$ and the statistic $\mathtt{pinv}_i$.

\begin{definition}
A numeral system is a notation for representing integers.
\end{definition}

\noindent There exist several types of numeral systems depending on the historical context and the geographical location. We develop a numeral system based on the colored permutation groups. Write $\gamma_k \centerdot \gamma_{k-1} \centerdot \dots \centerdot \gamma_1 \centerdot \gamma_0$ for the integer $\sum_{i=0}^k \gamma_i \mathsf{G}_i$, and $\langle\mathsf{G}, \mathsf{g}\rangle_k$ for the set
$$\langle\mathsf{G}, \mathsf{g}\rangle_k := \big\{\gamma_k \centerdot \gamma_{k-1} \centerdot \dots \centerdot \gamma_1 \centerdot \gamma_0 \ |\ \gamma_i \in \llbracket 0;\, \mathsf{g}_i \rrbracket \big\}.$$
Our numeral system stems from the following bijection.

\begin{theorem} \label{B2thm}
Let $n \geq 1$. The following map is bijective
$$g: \left. \begin{array}{ccc} \mathbb{U}_m \wr \mathfrak{S}_n & \rightarrow & \langle\mathsf{G}, \mathsf{g}\rangle_{n-1} \\ 
\pi & \mapsto & \mathtt{pinv}_1(\pi) \centerdot \mathtt{pinv}_2(\pi) \centerdot \dots \centerdot \mathtt{pinv}_n(\pi) \end{array} \right..$$
\end{theorem}

\noindent We prove Theorem \ref{B2thm} in Section \ref{SeNum}. The Lehmer code is based on the case $m=1$. It is a particular way to encode each permutation of $n$ numbers, and an instance of a scheme for numbering permutations. Moreover, Vajnovszki provided several permutation codes directly related to the Lehmer code \cite{Va}.

\begin{example}
From Example \ref{Ex1}, we have $\pi = \left( \begin{array}{cccc} 1 & 2 & 3 & 4 \\
2 & \zeta_1 1 & \zeta_4 4 & 3 \end{array} \right) = 1 \centerdot 3 \centerdot 9 \centerdot 0$ which, is equal to $945$ in decimal system.
\end{example}

\noindent Moreover, we obtain the generating function of the statistic number of pseudoinversions. Denote the $q$-analog of the number $i$ by $[i]_q := 1 + q + \dots + q^{i-1}$.

\begin{corollary} \label{Poin}
The generating function of the statistic $\mathtt{pinv}$ on $\mathbb{U}_m \wr \mathfrak{S}_n$ is
$$\sum_{\pi \in \mathbb{U}_m \wr \mathfrak{S}_n} q^{\mathtt{pinv}\,\pi} = \prod_{i=1}^n [mi]_q.$$
\end{corollary}

\begin{proof}
The bijection of Theorem \ref{B2thm} implies that every element of $\mathbb{U}_m \wr \mathfrak{S}_n$ has a unique representation in $\langle\mathsf{G}, \mathsf{g}\rangle_{n-1}$. Then, we have
$$\sum_{\pi \in \mathbb{U}_m \wr \mathfrak{S}_n} q^{\mathtt{pinv}\,\pi} = [\max \mathtt{pinv}_n +1]_q \dots [\max \mathtt{pinv}_2 +1]_q\, [\max \mathtt{pinv}_1 +1]_q = \prod_{i=1}^n [mi]_q.$$
\end{proof}

\section{Proof of Proposition \ref{PrNu}} \label{SeNuSy}

\noindent We provide a condition for a pair of integer sequences $(\mathsf{A}, \mathsf{a})$ to be a number system. 

\begin{lemma}\label{divrem}
Take a number system $(\mathsf{A}, \mathsf{a})$, and let $n = \alpha_k \centerdot \dots \centerdot \alpha_1 \centerdot \alpha_0$ be a nonnegative integer. Then, $$\alpha_k \mathsf{A}_k \leq n < (\alpha_k + 1)\mathsf{A}_k.$$
\end{lemma}

\begin{proof}
It is clear that $\alpha_k \mathsf{A}_k \leq n$. Suppose that $n \geq (\alpha_k + 1)\mathsf{A}_k$ which means $\sum_{i=0}^{k-1} \alpha_i\, \mathsf{A}_i \geq \mathsf{A}_k$. Then, there exist $\lambda_i$ such that $0 \leq \lambda_i \leq \alpha_i$ and $\lambda_{k-1} \centerdot \dots \centerdot \lambda_1 \centerdot \lambda_0 = 1 \centerdot \overbrace{0 \centerdot \dots \centerdot 0 \centerdot 0}^{k\ \text{times}}$. That contradicts the unicity of the representation.
\end{proof}

\begin{lemma}\label{ifonlyif}
Let $\mathsf{A} = (\mathsf{A}_i)_{i \in \mathbb{N}}$, $\mathsf{a} = (\mathsf{a}_i)_{i \in \mathbb{N}}$ be two sequences of positive integers. Then, the pair $(\mathsf{A}, \mathsf{a})$ is a number system if and only if $\mathsf{A}_0 = 1$ and 
$$\mathsf{A}_k = \sum_{i=0}^{k-1} \mathsf{a}_i \mathsf{A}_i + 1.$$ 
\end{lemma}

\begin{proof}
Suppose that $(\mathsf{A}, \mathsf{a})$ is a number system. It is obvious that we must have $\mathsf{A}_0 = 1$. From the second inequality of Lemma \ref{divrem}, we deduce that
$$\sum_{i=0}^k \mathsf{a}_i \mathsf{A}_i + 1 \leq (\mathsf{a}_k +1)\mathsf{A}_k \quad \text{i.e.} \quad
\sum_{i=0}^{k-1} \mathsf{a}_i \mathsf{A}_i + 1 \leq \mathsf{A}_k.$$
From the first inequality of Lemma \ref{divrem}, we deduce that the only possibility is 
$$\sum_{i=0}^{k-1} \mathsf{a}_i \mathsf{A}_i + 1 = \mathsf{A}_k.$$
Now, suppose that $\mathsf{A}_0 = 1$ and $\displaystyle \mathsf{A}_k = \sum_{i=0}^{k-1} \mathsf{a}_i \mathsf{A}_i + 1$. Then one can uniquely construct every positive integer by induction:
$$\text{if}\ n = \sum_{i=0}^k \alpha_i\, \mathsf{A}_i \text{, then}\ n+1 = \sum_{i=0}^k \alpha_i\, \mathsf{A}_i + 1 \in \langle\mathsf{A}, \mathsf{a}\rangle_{k+1}.$$
\end{proof}

\noindent We can now proceed to the proof of Proposition \ref{PrNu}:

\begin{proof}
From Lemma \ref{ifonlyif}, we deduce that $(\mathsf{A}, \mathsf{a})$ is a number system if and only if $\mathsf{A}_0 = 1$ and 
\begin{align*}
\mathsf{A}_k = & \sum_{i=0}^{k-1} \mathsf{a}_i \mathsf{A}_i + 1
= \mathsf{a}_{k-1} \mathsf{A}_{k-1} + \sum_{i=0}^{k-2} \mathsf{a}_i \mathsf{A}_i + 1
= \mathsf{a}_{k-1} \mathsf{A}_{k-1} + \mathsf{A}_{k-1}\\
= & (\mathsf{a}_{k-1} + 1) \mathsf{A}_{k-1}
= (\mathsf{a}_{k-1} + 1) (\mathsf{a}_{k-2} + 1) \dots (\mathsf{a}_0 + 1)
= \prod_{i=0}^{k-1}(1+ \mathsf{a}_i).
\end{align*}
\end{proof}

\section{Proof of Theorem \ref{B2thm}} \label{SeNum}

\noindent We prove that the map $g: \mathbb{U}_m \wr \mathfrak{S}_n \rightarrow \langle\mathsf{G}, \mathsf{g}\rangle_{n-1}$ is bijective, or in other words, for a number $\gamma_{n-1} \centerdot \gamma_{k-1} \centerdot \dots \centerdot \gamma_1 \centerdot \gamma_0$, there exists a colored permutation $\pi = \left( \begin{array}{cccc} 1 & 2 & \dots & n\\
\zeta_{k_1} \sigma(1) & \zeta_{k_2} \sigma(2) & \dots & \zeta_{k_n} \sigma(n) \end{array} \right)$ such that $$\mathtt{pinv}_1(\pi) \centerdot \mathtt{pinv}_2(\pi) \centerdot \dots \centerdot \mathtt{pinv}_n(\pi) = \gamma_{n-1} \centerdot \gamma_{k-1} \centerdot \dots \centerdot \gamma_1 \centerdot \gamma_0.$$

\begin{lemma} \label{LeCo}
Let $\gamma_i \in [0, \mathsf{g}_i]$. Then, there exist
\begin{itemize}
\item $k \in [0, m-1]$ such that $\gamma_i \in \big[k(i+1), k(i+1)+i\big]$,
\item and $\pi \in \mathbb{U}_m \wr \mathfrak{S}_n$ such that $k_{n-i} = k$.
\end{itemize}
\end{lemma}

\begin{proof}
On one side, $\mathtt{pinv}_{n-i}(\mathbb{U}_m \wr \mathfrak{S}_n) = [0, \mathsf{g}_i]$. On the other side, from its definition, we have $$\mathtt{pinv}_{n-i}\,\pi \in \big[k(i+1), k(i+1)+i\big] \ \Leftrightarrow \ k_{n-i} = k.$$ 
\end{proof}

\noindent We obtain the $\zeta_{k_i}$'s associated to $\gamma_{n-1} \centerdot \gamma_{k-1} \centerdot \dots \centerdot \gamma_1 \centerdot \gamma_0$ from Lemma \ref{LeCo}. It remains to determine the $\sigma(i)$'s. 

\smallskip

\noindent Let $s_i :=  \gamma_{n-i} - k(n-i+1)$. We need $n$ variables $x_1, \dots, x_n$ to replace each each of them in $x_1 > x_2 > \dots > x_n$ with the corresponding $\sigma(i)$. The notation $x_i \leftarrow \sigma(j)$ means we assign the value $\sigma(j)$ to the variable $x_i$. We implemente the following procedure from $n$ down to $1$:

\smallskip

\noindent \textsf{Put $x_n \leftarrow \sigma(n)$: we obtain $x_1 > x_2 > \dots > \sigma(n)$.\\
For $\sigma(n-1)$,
\begin{itemize}
\item[$\bullet$] If $s_{n-1}=1$, put $x_{n-1} \leftarrow \sigma(n-1)$: we obtain $x_1 > \dots > \sigma(n-1) > \sigma(n)$.
\item[$\bullet$] Else put $x_{n-1} \leftarrow \sigma(n)$, $x_n \leftarrow \sigma(n-1)$: we obtain 
$$x_1 > \dots > x_{n-2} > \sigma(n) > \sigma(n-1).$$
\end{itemize}
Recursively, for $\sigma(j)$,
\begin{itemize}
\item[$\bullet$] If $s_j=n-j$, put $x_{j} \leftarrow \sigma(j)$: we obtain $$x_1 > \dots > x_{j-1} > \sigma(j) > \overbrace{ \dots}^{\text{followed by}\ n-j\ \sigma(i)\text{'s}}.$$
\item[$\bullet$] Else, for every $i \in [n-j+1, n-s_j]$, put $x_{i-1} \leftarrow x_i$, and put $x_{n-s_j} \leftarrow \sigma(j)$: we obtain $$x_1 > \dots > x_{j-1} > \overbrace{\dots}^{\text{followed by}\ n-j-s_j\ \sigma(i)\text{'s}} > \sigma(j) > \overbrace{ \dots}^{\text{followed by}\ s_j\ \sigma(i)\text{'s}}.$$
\end{itemize}
At the end, we obtain a complete order of the $\sigma(i)$'s which means their values.}

\section{Application to Cryptography}

\noindent Here is an example of cryptographic coding based on Theorem \ref{B2thm}. Suppose that we want to encrypt a mail of $c$ characters written with $l$ symbols including letters, numbers and space. Consider the colored permutation group $\mathbb{U}_l \wr \mathfrak{S}_c$ and its corresponding number system $\langle\mathsf{G}, \mathsf{g}\rangle_{c-1}$. The mail can be considered as the element $x = \gamma_{c-1} \centerdot \dots \centerdot \gamma_1 \centerdot \gamma_0$ of $\langle\mathsf{G}, \mathsf{g}\rangle_{c-1}$ where $\gamma_{c-i}$ is the symbol order of the $i^{\text{th}}$ character. Choose a key $\kappa \in \mathbb{U}_l \wr \mathfrak{S}_c$. We define the encrypting function $\mathfrak{e}:\langle\mathsf{G}, \mathsf{g}\rangle_{c-1} \rightarrow \mathbb{U}_l \wr \mathfrak{S}_c$ by the function composition
$$\mathfrak{e}: \left. \begin{array}{ccccc} \langle\mathsf{G}, \mathsf{g}\rangle_{c-1} & \rightarrow & \mathbb{U}_l \wr \mathfrak{S}_c & \rightarrow & \mathbb{U}_l \wr \mathfrak{S}_c \\ 
m & \mapsto & g^{-1}(x) & \mapsto & g^{-1}(x) \circ \kappa \end{array} \right..$$
The crypted message is $y = \mathfrak{e}(x)$. We define the decrypting function $\mathfrak{d}:\mathbb{U}_l \wr \mathfrak{S}_c \rightarrow \langle\mathsf{G}, \mathsf{g}\rangle_{c-1}$ by
$$\mathfrak{d}: \left. \begin{array}{ccccc} \mathbb{U}_l \wr \mathfrak{S}_c & \rightarrow & \mathbb{U}_l \wr \mathfrak{S}_c & \rightarrow & \langle\mathsf{G}, \mathsf{g}\rangle_{c-1} \\ 
y & \mapsto & y \circ \kappa^{-1} & \mapsto & g(y \circ \kappa^{-1}) \end{array} \right..$$

\bibliographystyle{abbrvnat}

\end{document}